\documentclass[reqno]{amsart}

\pdfoutput=1

\usepackage[alphabetic]{amsrefs}
\usepackage{amsmath}
\usepackage{amsfonts}
\usepackage{amssymb}
\usepackage{enumerate}
\usepackage{amstext}
\usepackage{amsbsy}
\usepackage{amsopn}
\usepackage{bbm,amsthm}
\usepackage{amscd}
\usepackage[pdftex]{color}
\usepackage[pdftex]{graphicx}
\usepackage[all]{xy}
\usepackage{mathtools}
\usepackage{todonotes}
\usepackage{soul}

\usepackage{tikz}
\usetikzlibrary{intersections,positioning,patterns,arrows,decorations.markings}

\usepackage{todonotes}

\usepackage{amsxtra}
\usepackage{upref}
\usepackage{epstopdf}
\usepackage{graphicx,color}
\usepackage{hyperref}
\usepackage{longtable}
\usepackage[francais, english]{babel}

\makeatletter
\newsavebox{\@brx}
\newcommand{\llangle}[1][]{\savebox{\@brx}{\(\m@th{#1\langle}\)}%
  \mathopen{\copy\@brx\kern-0.5\wd\@brx\usebox{\@brx}}}
\newcommand{\rrangle}[1][]{\savebox{\@brx}{\(\m@th{#1\rangle}\)}%
  \mathclose{\copy\@brx\kern-0.5\wd\@brx\usebox{\@brx}}}
\makeatother

\DeclareFontFamily{OML}{rsfs}{\skewchar\font'177}
\DeclareFontShape{OML}{rsfs}{m}{n}{ <5> <6> rsfs5 <7> <8> <9> rsfs7
  <10> <10.95> <12> <14.4> <17.28> <20.74> <24.88> rsfs10 }{}
\DeclareMathAlphabet{\mathfs}{OML}{rsfs}{m}{n}

\newtheorem{theorem}{Theorem}
\newtheorem*{thm-homo-class}{Main Theorem}
\newtheorem{lemma}[theorem]{Lemma}
\newtheorem{proposition}[theorem]{Proposition}

\theoremstyle{definition}

\theoremstyle{remark}
\newtheorem{remark}[theorem]{\bf Remark}

\numberwithin{equation}{section}
\numberwithin{theorem}{section}

\newcommand{\intav}[1]{\mathchoice {\mathop{\vrule width 6pt height 3 pt depth  -2.5pt
\kern -8pt \intop}\nolimits_{\kern -6pt#1}} {\mathop{\vrule width
5pt height 3  pt depth -2.6pt \kern -6pt \intop}\nolimits_{#1}}
{\mathop{\vrule width 5pt height 3 pt depth -2.6pt \kern -6pt
\intop}\nolimits_{#1}} {\mathop{\vrule width 5pt height 3 pt depth
-2.6pt \kern -6pt \intop}\nolimits_{#1}}}

\newcommand{\intavl}[1]{\mathchoice {\mathop{\vrule width 6pt height 3 pt depth  -2.5pt
\kern -8pt \intop}\limits_{\kern -6pt#1}} {\mathop{\vrule width 5pt
height 3  pt depth -2.6pt \kern -6pt \intop}\nolimits_{#1}}
{\mathop{\vrule width 5pt height 3 pt depth -2.6pt \kern -6pt
\intop}\nolimits_{#1}} {\mathop{\vrule width 5pt height 3 pt depth
-2.6pt \kern -6pt \intop}\nolimits_{#1}}}

\newcommand{\vertiii}[1]{{\left\vert\kern-0.2ex\left\vert\kern-0.2ex\left\vert #1 
    \right\vert\kern-0.2ex\right\vert\kern-0.2ex\right\vert}}

\newcommand{\un}{\underline}

\newcommand{\ve}{\varepsilon}

\newcommand{\wh}{\widehat}

\newcommand{\vf}{\varphi}

\newcommand{\R}{\mathbb{R}}

\newcommand{\Z}{\mathbb{Z}}

\begin{document}

\title[Homoclinic classes on rank 1 manifolds]{Homoclinic classes of geodesic\\ flows on rank 1 manifolds}

\author{Yuri Lima and Mauricio Poletti}
\address{Departamento de Matem\'atica, Universidade Federal do Cear\'a (UFC), Campus do Pici,
Bloco 914, CEP 60455-760. Fortaleza -- CE, Brasil}
\email{yurilima@gmail.com}
\address{Departamento de Matem\'atica, Universidade Federal do Cear\'a (UFC), Campus do Pici,
Bloco 914, CEP 60455-760. Fortaleza -- CE, Brasil}
\email{mpoletti@mat.ufc.br}

\date{\today}
\thanks{We thank Katrin Gelfert, Davi Obata, Sergio Romaña, and the anonymous referee for valuable suggestions.
MP was supported by CNPq and Instituto Serrapilheira, grant number Serra-R-2211-41879.
YL was supported by CNPq and Instituto Serrapilheira, grant 
``Jangada Din\^{a}mica: Impulsionando Sistemas Din\^{a}micos na Regi\~{a}o Nordeste''.}

\begin{abstract}
Given a $C^{1+\beta}$ flow $\vf$ with positive speed on a closed smooth Riemannian manifold, we
code two homoclinically related $\vf$--invariant probabilities by an irreducible countable
topological Markov flow. As an application, we give
a proof using symbolic dynamics of the theorem of Knieper on the uniqueness
of the measure of maximal entropy \cite{Knieper-Rank-One-Entropy}
and theorems of Burns et al on the uniqueness of equilibrium states \cite{BCFT-2018}.
\end{abstract}

\maketitle


\section{Introduction}

Homoclinic classes are an efficient way of decomposing the dynamics of hyperbolic systems. Introduced by Newhouse in \cite{Newhouse-hyperbolic}
as generalizations of the basic sets considered by Smale in his spectral decomposition theorem 
\cite{Smale-differentiable}, homoclinic classes have been used in settings beyond uniform hyperbolicity, see e.g. \cite{Bonatti-Crovisier-Inventiones}.

Rodriguez Hertz et al introduced ergodic homoclinic classes of hyperbolic periodic
points, and studied the uniqueness of SRB measures for surface diffeomorphisms
\cite{HHTU-CMP}. Buzzi, Crovisier and Sarig introduced homoclinic classes of
measures and proved that the Markov partitions of
Sarig \cite{Sarig-JAMS} and Ben Ovadia \cite{Ben-Ovadia-2019} code
homoclinic classes by {\em irreducible} countable topological Markov shifts,
and each homoclinic class supports at most one equilibrium state
for each admissible potential, see \cite[Thm. 3.1, Cor. 3.3]{BCS-Annals}.
Buzzi, Crovisier and Lima studied homoclinic classes of measures
for 3--dimensional flows with positive speed \cite{BCL}.

We consider homoclinic classes for flows with positive speed in any dimension.
Our main result is the coding of two homoclinically related measures by an {\em irreducible}
topological Markov flow. Let $N$ be a $C^\infty$ closed Riemannian manifold.

\begin{thm-homo-class}\label{thm-coding-irreducible}
Let $X$ be a $C^{1+\beta}$ vector field with $X\neq 0$ everywhere, let
$\vf:N\to N$ be the flow generated by $X$, and let $\chi>0$. If $\mu_1,\mu_2$ are homoclinically related
$\chi$--hyperbolic probability measures, then there is an {\em irreducible} countable topological Markov flow
$(\Sigma_r,\sigma_r)$ and a H\"older continuous map $\pi_r:\Sigma_r\to N$ s.t.:
\begin{enumerate}[$(1)$]
\item $r:\Sigma\to\R^+$ is H\"older continuous and bounded away from zero and infinity.
\item $\pi_r\circ\sigma^t_r=\vf^t\circ\pi_r$ for all $t\in\R$.
\item $\pi_r[\Sigma^\#_r]$ has full measure with respect to $\mu_1$ and $\mu_2$.
\item Every $x\in N$ has finitely many pre-images in $\Sigma_r^\#$.
\end{enumerate} 
\end{thm-homo-class}

See Sections \ref{ss-symbolic} and \ref{ss-homoclinic} for the definitions of  topological Markov flow,
$\chi$--hyperbolicity, and homoclinic relation of measures.
Observe that the above theorem does not provide a single irreducible coding for the homoclinic class
of a hyperbolic measure as in \cite[Theorem 1.1]{BCL}, but it still provides interesting applications for 
geodesic flows over rank 1 manifolds on the uniqueness of equilibrium states, as we now explain. 

Let $M$ be a $C^\infty$ closed Riemannian manifold with nonpositive sectional curvature,
and let $\vf=\{\vf^t\}_{t\in\R}:T^1M\to T^1M$ be its geodesic flow. 
We say that $M$ is {\em rank $1$} if there is a vector $x\in T^1M$ without a parallel Jacobi field 
perpendicular to the flow direction. 
This assumption implies various geometrical and dynamical properties of $\vf$,
see Section \ref{ss-rank-1}. 
Let ${\rm HC}(\mathcal O)$ denote the homoclinic class of 
the periodic orbit $\mathcal O$, see Section \ref{ss-homoclinic} for the definition.
The next theorem follows from classical results on the theory, but we have decided
to state it as a theorem to stress its importance in the viewpoint of homoclinic classes.

\begin{theorem}\label{thm-homo-class}
Let $\vf$ be the geodesic flow over a $C^\infty$ closed rank $1$ Riemannian manifold.
If $\mathcal O$ is a hyperbolic periodic orbit, then ${\rm HC}(\mathcal O)=T^1M$. 
\end{theorem}

Using the Main Theorem and Theorem \ref{thm-homo-class}, we give proofs, for rank 1 manifolds, of the theorem of Knieper on the uniqueness of the measure
of maximal entropy \cite{Knieper-Rank-One-Entropy} and theorems of Burns et al on the 
uniqueness of equilibrium states for some classes of potentials \cite{BCFT-2018}.
Let $\psi^u$ denote the geometric potential of $\vf$ and let ${\rm Sing}$ be the {\em singular set} of $\vf$,
see Section \ref{ss-rank-1} for the definitions.

\begin{theorem}[Theorem A of \cite{BCFT-2018}]\label{ThmA-BCFT}
Let $M$ be a $C^\infty$ closed rank $1$ Riemannian manifold, and let $\psi:T^1M\to\mathbb R$
be H\"older continuous or of the form $\psi=q\psi^u$ for $q\in\R$.
If $P({\rm Sing},\psi)<P(\psi)$, then $\psi$ has a unique equilibrium state $\mu$.
This measure is hyperbolic and fully supported. 
\end{theorem}

Above, $P(\psi)$ and $P({\rm Sing},\psi)$ denote the topological pressure and topological pressure
restricted to ${\rm Sing}$ respectively.
We note that Theorem \ref{ThmA-BCFT} is not the full statement proved in \cite{BCFT-2018},
since we do not characterize $\mu$ as the weak--* limit of hyperbolic periodic orbits.
We also remark that $\mu$ is Bernoulli, due to \cite{LLS-2016} in dimension two and 
to \cite{Call-Thompson-2022,ALP} in any dimension.

For multiples of the geometric potential in surfaces, we also recover 
part of Theorem C of \cite{BCFT-2018}.

\begin{theorem}[Theorem C of \cite{BCFT-2018}]\label{ThmC-BCFT}
If $M$ is a $C^\infty$ closed rank $1$ surface, then its geodesic flow $\vf$ has a unique equilibrium state
$\mu_q$ for the potential $q\psi^u$ for each $q\in(-\infty,1)$. This measure
is hyperbolic and fully supported.
\end{theorem}

Again, we are not able to characterize $\mu$ as the weak--* limit of hyperbolic periodic orbits.
Finally, using the pressure gap $P({\rm Sing},0)<P(0)$ proved in \cite[Theorem B]{BCFT-2018},
we also recover Knieper's theorem on the uniqueness of the measure of maximal entropy.

\begin{theorem}[\cite{Knieper-Rank-One-Entropy}]\label{Thm-Knieper}
If $M$ is a $C^\infty$ closed rank $1$ Riemannian manifold, then its geodesic flow $\vf$
has a unique measure of maximal entropy. This measure
is hyperbolic and fully supported.
\end{theorem}

Let us stress that, in dimension two, the pressure gap is automatic, hence we provide
a self-contained proof of Knieper's theorem in dimension two, using
homoclinic classes and symbolic dynamics. 

We believe the approach developed in this paper can be applied to a wider class
of geodesic flows. Indeed, our techniques have already been extended to geodesic
flows on uniform visibility manifolds without conjugate points and
continuous Green bundles \cite{Wu}. We also mention the recent work of
Mamani and Ruggiero \cite{Mamani-Ruggiero}.

\section{Preliminaries}

\subsection{Symbolic dynamics}\label{ss-symbolic}

Let $\mathfs G=(V,E)$ be an oriented graph, where $V,E$ are the vertex and edge sets.
We denote edges by $v\to w$, and assume that $V$ is countable.

\medskip
\noindent
{\sc Topological Markov shift (TMS):} It is a pair $(\Sigma,\sigma)$
where
$$
\Sigma:=\{\text{$\Z$--indexed paths on $\mathfs G$}\}=
\left\{\un{v}=\{v_n\}_{n\in\Z}\in V^{\Z}:v_n\to v_{n+1}, \forall n\in\Z\right\}
$$
is the symbolic space and $\sigma:\Sigma\to\Sigma$, $[\sigma(\un v)]_n=v_{n+1}$, is the {\em left shift}. 
We endow $\Sigma$ with the distance $d(\un v,\un w):={\rm exp}[-\inf\{|n|\in\Z:v_n\neq w_n\}]$.
The {\em regular set} of $\Sigma$ is
$$
\Sigma^\#:=\left\{\un v\in\Sigma:\exists v,w\in V\text{ s.t. }\begin{array}{l}v_n=v\text{ for infinitely many }n>0\\
v_n=w\text{ for infinitely many }n<0
\end{array}\right\}.
$$

\medskip
We will sometimes omit $\sigma$ from the definition, referring to $\Sigma$ as a TMS.
We only consider TMS that are \emph{locally compact}, i.e.
for all $v\in V$ the number of ingoing edges $u\to v$ and outgoing edges $v\to w$ is finite.

\medskip
Given $(\Sigma,\sigma)$ a TMS, let $r:\Sigma\to(0,+\infty)$ be a continuous function.
For $n\geq 0$, let
$r_n=r+r\circ\sigma+\cdots+r\circ \sigma^{n-1}$ be the $n$--th {\em Birkhoff sum} of $r$,
and extend this definition for $n<0$
in the unique way such that the {\em cocycle identity} holds: $r_{m+n}=r_m+r_n\circ\sigma^m$, $\forall m,n\in\Z$.

\medskip
\noindent
{\sc Topological Markov flow (TMF):} The TMF defined
by $(\Sigma,\sigma)$ and \emph{roof function} $r$ is the pair $(\Sigma_r,\sigma_r)$ where
$\Sigma_r:=\{(\un v,t):\un v\in\Sigma, 0\leq t<r(\un v)\}$
and $\sigma_r:\Sigma_r\to\Sigma_r$ is the flow on $\Sigma_r$ given by
$\sigma_r^t(\un v,t')=(\sigma^n(\un v),t'+t-r_n(\un v))$, where
$n$ is the unique integer such that $r_n(\un v)\leq t'+t<r_{n+1}(\un v)$.
We endow $\Sigma_r$ with a natural metric $d_r(\cdot,\cdot)$,
called the {\em Bowen-Walters metric}, such that $\sigma_r$ is a continuous flow \cite{Bowen-Walters-Metric}.
The {\em regular set} of $(\Sigma_r,\sigma_r)$ is $\Sigma_r^\#=\{(\un v,t)\in\Sigma_r:\un v\in \Sigma^\#\}$.

\medskip
Similarly, we will sometimes omit $\sigma_r$ and refer to $\Sigma_r$ as a TMF.
The roof functions we consider will always be H\"older continuous, in which case
$\exists\kappa,C>0$ such that $d_r(\sigma_r^t(z),\sigma_r^{t}(z'))\leq C d_r(z,z')^\kappa$
for all $|t|\leq 1$ and $z,z'\in\Sigma_r$, see \cite[Lemma 5.8]{Lima-Sarig}.

\medskip
\noindent
{\sc Irreducible component:}
If $\Sigma$ is a TMS defined by an oriented graph
$\mathfs{G}=(V,E)$, its \emph{irreducible components} are the subshifts $\Sigma'\subset\Sigma$
defined over maximal subsets $V'\subset V$ satisfying the following condition:
$$\forall v,w\in V',\;\exists \un v\in \Sigma \text{ and } n\geq 1\text{ such that } v_0=v \text{ and } v_n=w.$$
An {\em irreducible component} $\Sigma'_r$ of $\Sigma_r$ is a set of the form
$\Sigma'_r=\{(\un v,t)\in\Sigma_r:\un v\in \Sigma'\}$
where $\Sigma'$ is an irreducible component of $\Sigma$.

\subsection{Homoclinic classes}\label{ss-homoclinic}

Let $M$ be a $C^\infty$ closed Riemannian manifold and $T^1M$ be its unit tangent bundle,
which is also a $C^\infty$ closed Riemannian manifold.
Let $\vf=\{\vf^t\}_{t\in\R}:T^1M\to T^1M$
be the geodesic flow on $M$, and $X$ be the vector field generating this flow.

\medskip
\noindent
{\sc $\chi$--hyperbolic measure:} A $\vf$--invariant probability measure $\mu$ on $T^1M$
is {\em $\chi$--hyperbolic} if for $\mu$--a.e. $x\in T^1M$ all the Lyapunov exponents are greater than $\chi$
in absolute value, except for the zero exponent in the flow direction.

\medskip
Let $\mu$ be a hyperbolic $\vf$--invariant probability measure. For $\mu$--a.e.
$x\in T^1M$,
$$
W^{ss}(x)=\left\{y\in T^1M:\limsup_{t\to+\infty}\tfrac{1}{t}\log d(\vf^t(x),\vf^t(y))<0\right\}
$$
denotes the {\em strong stable manifold} of $x$ and 
$$
W^{s}(x)=\bigcup_{t\in\R}\vf^t[W^{ss}(x)]
$$
denotes the {\em stable manifold} of $x$. We define similarly $W^{uu}(x)$ and $W^{u}(x)$
the {\em strong unstable} and {\em unstable} manifolds of $x$. 

Given a hyperbolic periodic orbit $\mathcal O$, we let
$W^{s/u}(\mathcal O)=W^{s/u}(x)$ denote
the stable/unstable manifold of $\mathcal O$, for any $x\in\mathcal O$. 

\medskip
\noindent
{\sc Homoclinic class of hyperbolic periodic orbit:} The {\em homoclinic class}
of a hyperbolic periodic orbit $\mathcal O$ is the set
$$
{\rm HC}(\mathcal O)=\overline{W^u(\mathcal O)\pitchfork W^s(\mathcal O)}.
$$  

\medskip
To simplify the notation, we will sometimes write $N:=T^1M$.

\medskip
\noindent
{\sc Homoclinic relation of measures \cite{BCS-Annals,BCL}:} 
We say that two ergodic hyperbolic measures $\mu,\nu$ are
\emph{homoclinically related} if for $\mu$--a.e. $x$ and $\nu$--a.e. $y$
there exist transverse intersections $W^{s}(x)\pitchfork W^{u}(y)\ne\emptyset$
and  $W^{u}(x)\pitchfork W^{s}(y)\ne\emptyset$, i.e. points $z_1\in W^{s}(x)\cap W^{u}(y)$
and $z_2\in W^{u}(x)\cap W^{s}(y)$ such that
$T_{z_1}N=T_{z_1}W^{s}(x)+T_{z_1}W^{u}(y)$
and $T_{z_2}N=T_{z_2}W^{u}(x)+T_{z_2}W^{s}(y)$.

\medskip
The homoclinic relation is an equivalence relation among ergodic hyperbolic measures,
see \cite[Prop. 10.1]{BCL}.

\subsection{Geodesic flows in nonpositive curvature}\label{ss-rank-1}

The rank of a vector $x\in T^1M$ is the dimension of
the space of parallel Jacobi fields along the geodesic defined by $x$.
Assume that $M$ is a {\em rank $1$} manifold: it has nonpositive sectional curvature
and there is a vector with rank 1, i.e. without parallel Jacobi fields perpendicular to the flow direction.

\medskip
\noindent
{\sc Regular and Singular sets:} The {\em regular set} of $\vf$ is defined by
$$
{\rm Reg}=\{x\in T^1M:x\text{ has rank }1\}.
$$
The {\em singular set} of $\vf$ is defined by
$$
{\rm Sing}=T^1M\setminus {\rm Reg}=\{x\in T^1M:x\text{ has rank }>1\}.
$$

\medskip
The sets ${\rm Reg}$ and ${\rm Sing}$ form a partition of $T^1M$, with ${\rm Reg}$
open and ${\rm Sing}$ closed. Although we use the same notation, the regular set of $\vf$
and the regular set of $\Sigma$ are not related; we maintain the notation because they 
are classical in their contexts.

Geodesic flows on rank 1 manifolds have (weak)
invariant directions called {\em Green bundles}, and invariant manifolds called {\em horospherical foliations}
at every point, which satisfy various properties summarized as follows.

 \begin{proposition}\label{prop.green.bundles}
There are continuous $d\vf$--invariant bundles $x\in T^1M\to E^{s/u}_x$
and continuous foliations $\{\mathcal F^{ss/uu}(x):x\in T^1M\}$ tangent to 
$E^{s/u}$ and invariant under $\vf$ satisfying the following properties:
\begin{enumerate}[{\rm (1)}]
\item $E^s,E^u$ are orthogonal to $X$ and ${\rm dim}(E^s)={\rm dim}(E^u)={\rm dim}(M)-1$.
\item $x\in {\rm Reg}$  if and only if $E^s_x\oplus \langle X_x\rangle\oplus E^u_x=T_x(T^1M)$.
\item ${\rm Reg}$ is dense in $T^1M$.
\item $\mathcal F^{s/u}(x):=\bigcup_{t\in\R}\mathcal F^{ss/uu}(\vf^t(x))$ is a $\vf$--invariant
connected manifold of dimension ${\rm dim}(M)$ and tangent to $E^{s/u}\oplus\langle X\rangle$.
\item There is a universal constant $C=C(M)$ such that if $d_x$
is the induced distance on $\mathcal F^s(x)$ then $d_{x}(\vf^t(x),\vf^t(y))\leq C d_x(x,y)$
for all $y\in \mathcal F^s(x)$ and all $t\geq 0$; a similar statement holds for $\mathcal F^u(x)$ and $t\leq 0$.
\item $\mathcal F^{s/u}(x)$ is dense in $T^1M$ for every $x\in T^1M$.
\end{enumerate}
\end{proposition}

The definitions of the Green bundles and horospherical foliations require a discussion on Jacobi fields
and Busemann functions; since these objects will not be further used in this paper, we refer the reader
to \cite[Section VI]{Eberlein} or \cite[Section 2.3]{Knieper-Handbook-Chapter}.
The proofs of (1)--(5) above can be found in \cite{Eberlein} and of (6) in \cite{Ballmann}.
For an orbit $\mathcal O$, define 
$\mathcal F^{s/u}(\mathcal{O}):=\mathcal F^{s/u}(x)$ for any $x\in\mathcal O$.
We also recall the definition of the geometric potential.

\medskip
\noindent
{\sc The geometric potential of $\vf$ \cite{Bowen-Ruelle-SRB}:}  The {\em geometric
potential} of $\vf$ is the function $\psi^u:T^1M\to\R$ defined as
$$
\psi^u(x)=-\tfrac{d}{dt}\Big|_{t=0}\log \det (d\vf^t_x|_{E^u_x})=-\lim_{t\to 0}\tfrac{1}{t}\log\det(d\vf^t_x|_{E^u_x}).
$$

\section{Proof of Main Theorem}

Let $X$ be a $C^{1+\beta}$ vector field on $N$ with $X\neq 0$ everywhere, and let
$\vf:N\to N$ be the flow generated by $X$. Fix $\chi>0$ and let $\mu_1,\mu_2$ be
homoclinically related $\chi$--hyperbolic probability measures.
If $N$ has dimension three, then we can apply \cite[Theorem 1.1]{BCL} directly.
Since this latter theorem is not available in higher dimension, we have to argue 
differently. For that, we combine techniques of \cite{Lima-Sarig,ALP,BCS-Annals,BCL}. 
Let $\mathcal O_1,\mathcal O_2,\ldots$ be the $\chi$--hyperbolic periodic orbits
homoclinically related to $\mu_1$ (and hence to $\mu_2$).\\

\noindent
\medskip
{\sc Claim:} There is a global Poincaré section $\Lambda$ such that if $f:\Lambda\to\Lambda$ is
the Poincaré return map and $\mathfs S$ is the {\em singular set} of $f$, then:
\begin{enumerate}[(1)]
\item $\Lambda$ is {\em adapted} for $\mu_1$ and $\mu_2$, i.e. 
\begin{equation}\label{subexp-convergence}
\lim_{|t|\to\infty}\tfrac{1}{t}\log d(\vf^t(x),\mathfs S)=0
\end{equation}
for $\mu_i$--a.e. $x\in \Lambda$, $i=1,2$.\footnote{This is equivalent to the projection of $\mu_i$ to $\Lambda$
being adapted in the notation of \cite{Lima-Sarig}.}
\item For every $n$, there exists a compact, $\vf$--invariant, transitive,
locally maximal, $\chi$--hyperbolic set $K_n$ that contains $\mathcal O_1,\ldots,\mathcal O_n$
and such that $K_n\cap \mathfs S=\emptyset$.
\end{enumerate}

\medskip
The singular set is $\mathfs S=\{x\in\Lambda:\{x,f(x),f^{-1}(x)\}\cap \partial\Lambda\neq\emptyset\}$.
\begin{proof}[Proof of the claim.]
Fix $x_i\in\mathcal O_i$ and $z_{ij}\in W^u(x_i)\pitchfork W^s(x_j)$, and let 
$\tau_{ij}\in\R$ such that $z_{ij}\in W^{uu}(x_i)\cap W^{ss}(\vf^{\tau_{ij}}(x_j))$. 
We construct a one-parameter family of global Poincaré sections $\Lambda_r$, where
$r$ varies in an interval $[a,b]$ and each $\Lambda_r$ is the union of finitely many codimension
1 disjoint balls $D_1(r),\ldots,D_k(r)\subset N$ centered at points $y_1,\ldots,y_k\in N$, each of them 
with radius $r$ and almost orthogonal to $X$. The details of the construction can be found in
\cite[Section 2]{Lima-Sarig} for three-dimensional flows and in \cite[Section 10]{ALP} for any dimension.
We can assume that the roof functions of the Poincaré return maps to $\Lambda_r$
are all larger than some $\ve_0>0$.
As the radius $r$ varies, the boundary of $\Lambda_r$ varies as well. 
Applying a double counting argument and the Borel-Cantelli
lemma, for Lebesgue almost every choice $r\in[a,b]$ the section $\Lambda_r$ satisfies (1) above, see 
\cite[Theorem 2.8]{Lima-Sarig} for details.
Hence we focus on showing that (2) holds in a set of parameters $r\in [a,b]$ of positive measure.
We will prove this using parameter selection.
Observe that if $K$ is $\vf$--invariant, then $K\cap\mathfs S=\emptyset$ if and only if $K\cap \partial \Lambda_r=\emptyset$.

Write $z_{ii}=x_i$, and let $\mathcal O_{ij}$
denote the orbit of $z_{ij}$. Consider the countable union of orbits
$\bigcup_{i,j}\mathcal O_{ij}$. Clearly, the set of parameters
$r\in[a,b]$ such that $\left(\bigcup_{i,j}\mathcal O_{ij}\right)\cap\partial\Lambda_r\neq\emptyset$
is countable. Below, we take $r$ in the complement of this set.
We use the following notation: given two functions $g(\delta),h(\delta)$,
write $g=O_n(h)$ if there is a constant $C>0$ that depends on $n$ and $\delta_0>0$ 
such that $|g(\delta)|\leq C|h(\delta)|$, $\forall|\delta|<\delta_0$.

Fix $n$ and $\delta>0$. In the sequel, we construct a $\delta$--neighborhood 
of $\bigcup_{1\leq i,j\leq n}\mathcal O_{ij}$. 
Let $C=C(n)>0$ such that
$$
\begin{array}{r}
d(\vf^{-t}(x_i),\vf^{-t}(z_{ij}))\leq Ce^{-\chi t}\\
d(\vf^{t+\tau_{ij}}(x_j),\vf^{t}(z_{ij}))\leq Ce^{-\chi t}
\end{array}, \ \ \forall 1\leq i,j\leq n,\forall t\geq 0.
$$
Let $t_{ij}=O_n(|\log\delta|)$ positive
such that $\vf^{-t_{ij}}(z_{ij})\in B_\delta(\mathcal O_i)$ and $\vf^{t_{ij}}(z_{ij})\in B_\delta(\mathcal O_j)$.
The set 
$$
Z_n=\bigcup_{1\leq i\leq n}\mathcal O_i\cup \bigcup_{1\leq i,j\leq n}\{\vf^t(z_{ij}):|t|\leq t_{ij}\}
$$
is the union of finitely many pieces of orbits and $\bigcup_{1\leq i,j\leq n}\mathcal O_{ij}\subset B_\delta(Z_n)$.
Arguing as in \cite[Lemma 3.11]{BCS-Annals}, we can construct inside $B_\delta(Z_n)$
a compact, $\vf$--invariant, transitive, locally maximal, $\chi$--hyperbolic set that contains
$\mathcal O_1,\ldots,\mathcal O_n$. Thus, we estimate the measure of the
set of $r\in[a,b]$ such that $B_\delta(Z_n)\cap\partial \Lambda_r\neq\emptyset$.

Denoting by $|\mathcal O_i|$ the length of $\mathcal O_i$, the total length of $Z_n$ is
$$
\sum_{1\leq i\leq n}|\mathcal O_i|+\sum_{1\leq i,j\leq n} 2t_{ij}=O_{n}(|\log\delta|).
$$
Recalling that $\ve_0$ is a lower bound for the roof function of $\Lambda_r$, the
set $Z_n\cap \Lambda_r$ has at most $\tfrac{1}{\ve_0}O_{n}(|\log\delta|)=O_n(|\log\delta|)$ elements.
Then the intersection
$B_\delta(Z_n)\cap\Lambda_b$ is contained in $O_n(|\log\delta|)$ balls of radius $2\delta$.
The set of parameters ${\rm Bad}(n)=\{r\in [a,b]:B_\delta(Z_n)\cap\partial\Lambda_r\neq\emptyset\}$
is thus contained in the union of $O_n(|\log\delta|)$ intervals of length $4\delta$, and so has
Lebesgue measure $O_n(\delta|\log\delta|)$. Since $\lim_{\delta\to 0}\delta\log\delta=0$,
we can take $\delta_n$ such that ${\rm Leb}[{\rm Bad}(n)]<\ve/n^2$ for small $\ve>0$
and so the complement $[a,b]\backslash\bigcup_{n\geq 1} {\rm Bad}(n)$ has positive Lebesgue measure.

In summary, we can choose $r\in[a,b]$ such that:
\begin{enumerate}[$\circ$]
\item Condition (\ref{subexp-convergence}) is satisfied: the space of $r\in[a,b]$ satisfying
it has full measure.
\item $\left(\bigcup_{i,j}\mathcal O_{ij}\right)\cap\partial\Lambda_r=\emptyset$:
the space of $r\in[a,b]$ satisfying it is the complement of a countable set, hence
has full measure.
\item $B_{\delta_n}(Z_n)\cap\partial \Lambda_r=\emptyset$ for all $n\geq 1$: 
the space of $r\in[a,b]$ satisfying it has positive measure.
\end{enumerate}
This concludes the proof of the claim.
\end{proof}

Once the section $\Lambda=\Lambda_r$ is chosen, apply \cite[Theorem 10.1]{ALP}
to construct a TMF $(\Sigma_r,\sigma_r)$ and
a H\"older continuous map $\pi_r:\Sigma_r\to N$ satisfying (1)--(4) in Theorem \ref{thm-coding-irreducible},
with the exception that $\Sigma_r$ might not be irreducible.
Since $K_n\cap \mathfs S=\emptyset$ and $K_n$ is compact and $\chi$--hyperbolic,
\cite{ALP} implies that $K_n\subset \pi_r[\Sigma_r]$. Since $K_n$ is uniformly hyperbolic,
we actually have $K_n\subset \pi_r[\Sigma_r^\#]$.

The final step is to find an irreducible component $\Sigma'_r$ of $\Sigma_r$
that lifts both $\mu_1$ and $\mu_2$. For that, we proceed as in \cite[Lemma 10.4]{BCL}:
\begin{enumerate}[$\circ$]
\item For each $n\geq 1$, there is an invariant, compact, transitive set $X_n\subset\Sigma_r^\#$ that lifts $K_n$.
\item Since $\mathcal O_1$ has finitely many lifts in $\Sigma_r^\#$, the sequence $X_1,X_2,\ldots$
has a subsequence $X_{n_1},X_{n_2},\ldots$ containing a same lift $(\un v,t)$ of $\mathcal O_1$.
Take $\Sigma_r'$ to be the irreducible component of $(\un v,t)$. Then $\mathcal O_1,\ldots$ all
lift to periodic orbits in $\Sigma_r'$. 
\end{enumerate}
Now proceed as in \cite[Theorem 1.1]{BCL} to lift generic points for $\mu_1$ and $\mu_2$ to $\Sigma'_r$. 
This completes the proof of the Main Theorem.

\section{Applications to geodesic flows on rank 1 manifolds}

Let $M$ be a rank 1 manifold. In this section, we prove Theorems \ref{thm-homo-class}--\ref{Thm-Knieper}.
We will use a classical result of Ballmann, Brin and Eberlein \cite[Prop. 3.10]{BBE-Annals}.
Call $x\in T^1M$ {\em uniformly recurrent} if for every neighborhood $U\subset T^1M$ of $x$ it holds
$$
\liminf_{T\to\infty}\frac{1}{T}\int_0^T 1_{U}(\vf^t(x))dt>0,
$$
where $1_U$ is the characteristic function of $U$. 
By the Birkhoff ergodic theorem, if $\mu$ is $\vf$--invariant then 
$\mu$--a.e. $x$ is uniformly recurrent.

\begin{lemma}[\cite{BBE-Annals}]\label{lemma-BBE}
If $x\in{\rm Reg}$ is uniformly recurrent then $W^{s/u}(x)=\mathfs F^{s/u}(x)$.
In particular, if $\mathcal O$ is a hyperbolic periodic orbit then
$W^{s/u}(\mathcal{O})=\mathcal F^{s/u}(\mathcal O)$.
\end{lemma}

\begin{proof}[Proof of Theorem \ref{thm-homo-class}]
Let $\mathcal O$ be a hyperbolic periodic orbit.
By Lemma \ref{lemma-BBE} and Proposition~\ref{prop.green.bundles}(6), we conclude that 
$W^{s/u}(\mathcal O)=\mathcal F^{s/u}(\mathcal O)$ is dense in $T^1M$.
Fix $y\in {\rm Reg}$, and let $U\subset{\rm Reg}$ be an open set containing $y$.
By Proposition~\ref{prop.green.bundles}(2) and (4), we have that $y\in \mathcal{F}^{s}(y)\pitchfork  \mathcal{F}^{u}(y)$. By the continuity of $\mathcal F^{s/u}$, there is $\varepsilon>0$ such that if $z_s,z_u\in B_\varepsilon(y)$ then   $(\mathcal{F}^{s}(z_s)\pitchfork  \mathcal{F}^{u}(z_u))\cap U\neq \emptyset$. 

Since $W^{s/u}(\mathcal O)$ is dense, there is $z_{s/u}\in W^{s/u}(\mathcal O)\cap B_\ve(y)$.
In particular, we have $W^{s/u}(\mathcal O)=\mathcal{F}^{s/u}(z_{s/u})$. By the choice of $\ve$,
it follows that $W^u(\mathcal{O}) \pitchfork W^s(\mathcal O)$
at some point $z\in U$. This proves that $y\in {\rm HC}(\mathcal O)$. Since $y\in{\rm Reg}$ is arbitrary,
${\rm HC}(\mathcal O)\supset {\rm Reg}$. Finally, since ${\rm Reg}$ is dense in $T^1M$,
we get that ${\rm HC}(\mathcal O)=T^1M$.
\end{proof}

\begin{proof}[Proof of Theorem \ref{ThmA-BCFT}]
Let $\psi$ be H\"older continuous or of the form $\psi=q\psi^u$ with $q\in\R$. 
In particular, $\psi$ is continuous, hence
the existence of an equilibrium state is guaranteed by the $C^\infty$ regularity
of $\vf$ \cite{Newhouse-Entropy}.
For the uniqueness, let $\mu_1,\mu_2$ be two ergodic equilibrium states. Since
$P({\rm Sing},\psi)<P(\psi)$, we have $\mu_1({\rm Reg})=\mu_2({\rm Reg})=1$
and so $\mu_1,\mu_2$
are hyperbolic. Since $\mu_1,\mu_2$ are ergodic, we can take $\chi>0$ small so that $\mu_1,\mu_2$
are $\chi$--hyperbolic. We claim that $\mu_1$ and $\mu_2$ are homoclinically 
related. To see this, let $y_1,y_2\in{\rm Reg}$ be uniformly recurrent. Since $W^{s/u}(y_i)=\mathcal F^{s/u}(y_i)$ is dense in $T^1M$
(Lemma \ref{lemma-BBE} and Proposition~\ref{prop.green.bundles}(6)), the same argument in the proof of Theorem~\ref{thm-homo-class} shows that $W^{u}_{\rm loc}(y_1)\pitchfork W^{s}(y_2)\neq\emptyset$.
Regular uniformly recurrent points are generic for $\mu_1,\mu_2$, hence the claim follows. 

Applying Theorem \ref{thm-coding-irreducible} to $\mu_1$ and $\mu_2$, we get
an irreducible TMF $(\Sigma_r,\sigma_r)$ and a Hölder continuous map $\pi_r:\Sigma_r\to T^1M$
such that $\mu_i[\pi_r(\Sigma_r^\#)]=1$ for $i=1,2$. Therefore,
$\mu_1,\mu_2$ lift to ergodic measures $\widehat\mu_1,\widehat\mu_2$ on $\Sigma_r$. These measures are equilibrium states of the potential $\widehat\psi=\psi\circ \pi_r$. 

\medskip
\noindent
{\sc Claim:} $\wh\psi$ is Hölder continuous. 

\begin{proof}[Proof of the claim.]
When $\psi$ is H\"older continuous,
$\widehat\psi=\psi\circ \pi_r$ is the composition of two H\"older maps, hence Hölder continuous.

When $\psi=q\psi^u$ for some $q\in\R$, we show that
$(\un v,t)\in\Sigma_r\mapsto E^u_{\pi_r(\un v,t)}$ is Hölder continuous, which obviously implies
that $\wh\psi$ is Hölder continuous. It is enough to show that $\un v\in\Sigma\mapsto E^u_{\pi_r(\un v,0)}$
is Hölder continuous, since by $d\vf$--invariance this implies that $(\un v,t)\in\Sigma_r\mapsto E^u_{\pi_r(\un v,t)}$
is Hölder continuous.

As in the proof of the Main Theorem, let $f:\Lambda\to\Lambda$ be the Poincaré
return map. The map 
$\mathfs F:\un v\in \Sigma\to F^u_x\subset T_x\Lambda$, where
$F^u_x$ is the unstable direction for $f$ at $x=\pi_r(\un v,0)$, is H\"older continuous 
\cite[Proposition 7.7]{ALP}.
Now let
$\mathfrak p_x:T_x\Lambda\to X_x^\perp$ be the orthogonal projection. 
Such map exists and is an isomorphism because 
both $X_x^\perp,T_x\Lambda$ have dimension ${\rm dim}(T^1M)-1$
and $T_x\Lambda$ is almost orthogonal to $X_x$. We have 
$E^u_x=\mathfrak p_x(F^u_x)$.  
Since $x\in T^1M\mapsto X_x^\perp$ and $x\in \Lambda\mapsto T_x\Lambda$ are $C^\infty$,
the map $\mathfs P:x\in \Lambda\mapsto\mathfrak p_x$ is $C^\infty$. Therefore
$\un v\in\Sigma\mapsto E^u_{\pi_r(\un v,0)}$, being the composition $\mathfs P\circ \mathfs F$, 
is Hölder continuous.
\end{proof}

The measures $\widehat\mu_1,\widehat\mu_2$ project to ergodic $\sigma$--invariant
probability measures $\widehat\nu_1,\widehat\nu_2$ on the irreducible component $\Sigma$
which are equilibrium states
of the H\"older continuous potential $\widehat\psi_r-P_{\rm top}(\psi)r$ where
$\widehat\psi_r(\un v)=\int_0^{r(\un v)}\wh\psi(\un v,t)dt$, see e.g. \cite[Proposition 6.1]{Parry-Pollicott-Asterisque}.
By \cite[Theorem 1.1]{Buzzi-Sarig}, we conclude that $\widehat\nu_1=\widehat\nu_2$,
and so $\mu_1=\mu_2$.

Finally, we show that the unique equilibrium state $\mu$ is fully supported. The proof is the same
of \cite[Corollary 3.3]{BCS-Annals}. Using the same notation of
the previous paragraphs, $\wh\nu$ has full support in $\Sigma$ by \cite{Buzzi-Sarig} and so
$\wh\mu$ has full support in $\Sigma_r$. This implies that ${\rm supp}(\mu)=\overline{\pi_r(\Sigma_r)}$, 
so it is enough to show that $\pi_r(\Sigma_r)$ is dense in $T^1M$. Let $\mathcal O$ be a hyperbolic
periodic orbit homoclinically related to $\mu$.\footnote{The existence of this orbit is consequence of
Katok's horseshoe theorem for flows; another way to obtain this is by the symbolic coding of \cite{ALP}.}
We thus have ${\rm supp}(\mu)\subset {\rm HC}(\mathcal O)=T^1M$, hence it is enough to show that
$\pi_r(\Sigma_r)$ is dense in $\{W^u(\mathcal O)\pitchfork W^s(\mathcal O)\}$. 
The proof of this fact in \cite[Corollary 3.3]{BCS-Annals} uses \cite[Proposition 3.7]{BCS-Annals},
which works equally well in our context. Hence we conclude that 
${\rm supp}(\mu)=\overline{\pi_r(\Sigma_r)}={\rm HC}(\mathcal O)=T^1M$.
\end{proof}

\begin{remark}
When $M$ is a surface and $\psi=q\psi^u$, \cite{LLS-2016} can be used to prove uniqueness
without using the Hölder continuity of $\wh\psi$. In this cited work it is proved that $\wh\psi_r$ above
is cohomologous to the Hölder continuous function 
$\un v\in\Sigma\mapsto -\log\det(df|_{F^u_{\pi_r(\un v,0)}})$,
 see \cite[Lemmas 8.1 and 8.2]{LLS-2016}.
We note that the same proof applies in any dimension.
\end{remark}

\begin{proof}[Proof of Theorem \ref{ThmC-BCFT}]
The proof is the same of \cite[Theorem C]{BCFT-2018}. We include it for
completeness.
Let $M$ be a surface and $\psi=q\psi^u$ for $q<1$.
Inside ${\rm Sing}$ all Lyapunov exponents are zero. Let $\mu$ be a $\vf$--invariant measure 
supported on ${\rm Sing}$. By the Ruelle inequality, $h(\mu)=0$.
Also, since $\psi^u=0$ on ${\rm Sing}$, it follows that
$$
h(\mu)+q\int \psi^u d\mu = 0.
$$ 
Therefore $P({\rm Sing},\psi)=0$.
On the other hand, if $\lambda$ is the Lebesgue measure, then by Pesin equality
$$
h(\lambda)+q\int\psi^u d\mu = (1-q)h(\lambda)>0.
$$
Therefore $P(\psi)>0=P({\rm Sing},\psi)$.
Now apply Theorem \ref{ThmA-BCFT}.
\end{proof}

\begin{proof}[Proof of Theorem \ref{Thm-Knieper}]
When $M$ is a surface we have, as in the previous proof, that
$P({\rm Sing},0)=0$. Since $P(0)=h_{\rm top}(\vf)>0$, we get the pressure gap $P({\rm Sing},0)<P(0)$.
Applying Theorem \ref{ThmA-BCFT}, we conclude the
uniqueness of the measure of maximal entropy. In higher dimension,
we invoke \cite[Theorem B]{BCFT-2018} to get the pressure gap $P({\rm Sing},0)<P(0)$.
Again, applying Theorem \ref{ThmA-BCFT}, we obtain the
uniqueness of the measure of maximal entropy.
\end{proof}

\bibliographystyle{alpha}
\bibliography{bibliography}{}

\end{document}